\documentclass[12pt]{amsart}

\usepackage{amsmath}
\usepackage{amssymb}
\usepackage{graphicx}

\def\lt{\left}                
\def\rt{\right}
\newcommand{\R}{{\mathbb R}}       

\newcommand{\N}{{\mathbb N}}       %

\newcommand{\diam}{{\rm diam}}

\newcommand{\ra}{\rightarrow}

\newcommand{\supp}{{\operatorname{spt}}}

\newcommand{\ve}{{\varepsilon}}

\newcommand{\loc}{{\rm loc}}

\newcommand{\stm}{\setminus}

\newcommand{\Rd}{{\mathbb{R}^d}}
\newtheorem{theorem}{Theorem}[section]
\newtheorem{lemma}[theorem]{Lemma}

\newtheorem*{lemma*}{Lemma}

\theoremstyle{definition}

\theoremstyle{remark}
\newtheorem{rem}[theorem]{\bf Remark}

\numberwithin{equation}{section}

\newcommand{\brem}{\begin{rem}}
\newcommand{\erem}{\end{rem}}

\begin{document}

\title{A note on weak convergence of singular integrals in metric spaces}

\author{Vasilis Chousionis}
\address{Department of Mathematics \\ University of Illinois \\ 1409
  West Green St. \\ Urbana, IL 61801}
\email{vchous@math.uiuc.edu}

\author{Mariusz Urba\'nski}
\address{Department of Mathematics \\University of North Texas\\ General Academics Building 435
  \\ 1155 Union Circle \#311430 \\ Denton, TX 76203-5017}
\email{urbanski@unt.edu}



\begin{abstract}We prove that in any metric space $(X,d)$ the singular integral
operators
\begin{equation*}
T^k_{\mu,\ve}(f)(x)=\int_{X\setminus
B(x,\varepsilon )}k(x,y)f(y)d\mu (y) .
\end{equation*}
converge weakly in some dense subspaces of $L^2(\mu)$ under minimal regularity
assumptions for the measures and the kernels.
\end{abstract}
\maketitle


\section{Introduction}
A Radon measure on a metric space $(X,d)$ has $s$-growth if
there exists some constant $c_\mu$ such that $\mu(B(x,r))\leq c_\mu r^{s}$
for all $x\in X$, $r>0$. 

We say that
$k(\cdot,\cdot):X\times X\setminus\{(x,y)\in X \times X :x=y\}\rightarrow\R$ is
an $s$-dimensional kernel if there exists a
constant $c>0$ such that for all $x,y\in X$, $x\neq y$:
\begin{equation*}
\label{defcz}
\begin{split}
&|k(x,y)|  \leq c \, d(x,y)^{-s}.
\end{split}
\end{equation*}
The kernel $k$ is antisymmetric if $k(x,y)=-k(y,x)$ for all distinct $x,y \in X$.

Given a positive Radon measure $\nu$ on $X$ and an $s$-dimensional kernel $k$, we define
\begin{equation*}
T^k\nu(x) := \int k(x,y)\,d\nu(y), \qquad{x\in X \setminus\supp\nu}.
\end{equation*}
This integral may not converge when $x\in\supp\nu$. For this reason, we
consider the following $\ve$-truncated operators $T^k_\ve$, $\ve>0$:
$$T^k_\ve\nu(x) := \int_{d(x,y)>\ve} k(x,y)\,d\nu(y), \qquad{x\in X}.$$

Given a fixed positive Radon measure $\mu$ on $X$ and $f\in
L^1_{\loc}(\mu)$, we write
$$T^k_\mu f(x) := T^k(f\,\mu)(x),\qquad x\in X\setminus\supp(f\,\mu),$$
and
$$T^k_{\mu,\ve} f(x) := T^k_\ve(f\,\mu)(x).$$

Concerning the limit properties of the operators $T^k_{\mu,\ve}$ one can ask if the limit, the so called principal value of $T$,
\begin{equation*}
\lim_{\varepsilon \rightarrow 0}T^k_{\mu,\ve}(f)(x),
\end{equation*}
exists $\mu$ almost everywhere. When $\mu$ is the
Lebesgue measure in $\Rd$, and $k$ is a standard
Calder\'{o}n-Zygmund kernel, due to cancellations and the denseness
of smooth functions in $L^1$, the principal values exist almost
everywhere for $L^1$-functions. For more general measures, the
question is more complicated. Let $n$ be an integer, $0 <n <d$, and
consider the coordinate Riesz kernels
\begin{equation*}
R_i^{n} (x)=\frac {x_i}{|x|^{n+1}}\text{ for } i=1,\dots,d.
\end{equation*}
Tolsa proved in \cite{T2} that if $E \subset \R^d$ has finite $n$-dimensional Hausdorff measure $\mathcal{H}^n$ the principal values
\begin{equation*}
\lim_{\varepsilon \rightarrow 0}\int_{E\setminus
B(x,\varepsilon )}\frac {x_i-y_i}{ |x-y|^{m+1}}d \mathcal{H}^n (y)
\end{equation*}
exist $\mathcal{H}^n$ almost everywhere in $E$ if and only if the set $E$ is
$n$-rectifiable i.e. if there exist $n$-dimensional Lipschitz
surfaces $M_{i}$, $i\in \N$, such that
\begin{equation*}
\mathcal{H}^{n}(E \setminus \cup_{i=1}^\infty M_{i})=0. 
\end{equation*}
Mattila and Preiss had obtained the same result earlier, in \cite{MPr} under some stronger assumptions for the set $E$. It becomes obvious that the existence of principal values is deeply related to the geometry of the set $E$.

Assuming $L^2(\mu)$-boundedness for the operators $T^k_\mu$ one could have expected that more could be deduced about the structure of $\mu$ and the existence of principal values, but this is a hard and, in a large extent, open problem. Dating from 1991 the David-Semmes conjecture, see \cite{DS}, asks if the $L^2(\mu)$-boundedness of the operators associated with the $n$-dimensional Riesz kernels suffices to imply $n$-uniform rectifiabilty, which can be thought as a quantitative version of rectifiability. In the very recent deep work \cite{NToV}, Nazarov, Tolsa and Volberg resolved the conjecture in the codimension 1 case, that is for  $n=d-1$. Mattila, Melnikov and Verdera in \cite{MMV}, using a special symmetrization property of the Cauchy kernel, had earlier proved the conjecture in the case of $1$-dimensional Riesz kernels. For all other dimensions and for other kernels few things are known. In fact, there are several examples of kernels whose boundedness does not imply rectifiability, see \cite{C}, \cite{D3} and \cite{hu}. For some recent positive results involving other kernels see \cite{CMPT}.

Let $\mu$ be a finite Radon measure and let $k$ be an antisymmetric kernel in a complete metric space $(X,d)$ where the Vitali covering theorem holds for $\mu$ and the family of closed balls defined by $d$. Mattila and Verdera in \cite{MV} showed that in this case the $L^2(\mu)$-boundedness of the
operators $T^k_{\mu,\ve}$ forces them to converge weakly
in $L^2(\mu)$. This means that there exists a bounded linear
operator $T^k_{\mu}:L^2(\mu)\rightarrow L^2(\mu)$ such that for all
$f,g \in L^2(\mu)$,
\begin{equation*}
\lim _{\varepsilon \rightarrow 0} \int T^k_{\mu,\ve}(f)(x)g(x)d\mu (x)=\int T^k_{\mu}(f)(x)g(x)d\mu (x).
\end{equation*}
Furthermore notions of weak convergence have been recently used by Nazarov, Tolsa and Volberg in \cite{NToV}.

Motivated by these developments it is natural to ask if limits of this type
might exist if we remove the very strong $L^2$-boundedness
assumption. We prove that the operators $T^k_{\mu,\ve}$ converge
weakly in dense subspaces of $L^2(\mu)$ under minimal assumptions for the measures and the kernels in general metric spaces. 
Denote by $\mathcal{X}_{B}$ the space of all finite linear combinations of
characteristic functions of balls in $X$,
$$\mathcal{X}_{B}=\left\{\sum_{i=1}^n a_i\chi_{B(z_i,r_i)}: \, n \in \N, \, a_i\in \R, \,z_i \in X, \, r_i>0\right\}.$$
Whenever Vitali's covering theorem holds for the closed balls in $(X,d)$ the space $\mathcal{X}_{B}$ is dense in $L^2(\mu)$. When $X=\Rd$  Vitali's covering theorem holds for any Radon measure $\mu$ and the closed balls defined by various metrics (including the standard $d_p$ metrics for $1\leq p\leq \infty$) as a consequence of Besicovitch's covering theorem, see \cite[Theorem 2.8]{M}. 
Furthermore Vitali's covering theorem holds for any metric space $(X,d)$ whenever $\mu$ is doubling, that is when there exists some constant $C$ such that for all balls $B$, $\mu(2 B) \leq C \mu (B)$, see \cite[Section 2.8]{F}.
\begin{theorem}
\label{main}
Let $\mu$ be a finite Radon measure with $s$-growth and $k$ an antisymmetric $s$-dimensional kernel on a metric space $(X,d)$. If the Vitali Covering theorem holds for the closed balls in $(X,d)$ then there exists subsets $\mathcal{X}_{B}' \subset \mathcal{X}_{B}$ which are dense in $L^2(\mu)$ and the weak limits
$$\lim_{\ve \ra 0} \int T^k_{\mu,\ve} f(x) \,g(x) d \mu(x)$$
exist for all $f,g \in  \mathcal{X}_{B}'$.
\end{theorem}

Until now Theorem \ref{main} was only known for measures with $(d-1)$-growth in $\Rd$ under some smoothness assumptions for the kernels, see
\cite{CM}. We thus extend the result from \cite{CM} to measures with $s$-growth for arbitrary $s$ in metric spaces where Vitali's covering theorem holds for the family of closed balls without requiring any smoothness for the kernels. Our proof follows a completely different strategy using  an ``exponential growth" lemma for probability measures on intervals and is self contained (unlike the proof from \cite{CM} which depends on several $L^2(\nu)$ to $L^2(\mu)$ boundedness results for separated measures $\nu$ and $\mu$).

Recall that if $k$ is the $(d-1)$-dimensional Riesz kernel in $\Rd$ and $\mu$ has $(d-1)$-growth and is $(d-1)$ purely unrectifiable, that is  $\mu(E)=0$ for all $(d-1)$-rectifiable sets $E$, the principal values  diverge $\mu$ almost everywhere and the weak convergence in $L^2(\mu)$ fails. On the other hand it is of interest that weak convergence in the sense of Theorem \ref{main} holds as it holds for any $s$-dimensional antisymmetric kernel and any finite measure with $s$-growth.

\section{Proof of Theorem \ref{main}}
We first prove the following lemma about exponential growth of
probability measures on compact intervals. It is motivated by a
similar result proved in \cite{uzd}. Here ${\rm Leb}$ stands for
the Lebesgue measure on the real line and $|I|$ denotes the length of an interval $I \subset \R$. 

\begin{lemma}
\label{exp} For every integer $\lambda >2$ the following holds. Let
$\nu$ be a probability Borel measure on a compact interval $\Delta
\subset \R$. Then for every interval $I \subset \Delta$ there exists a
subset $I'(\lambda) \subset I$ such that $Leb(I'(\lambda) ) > |I|(1
-3( \lambda^{-1} + \lambda^{-2} + \dots ))$ and for every $t \in
I'(\lambda)$, 
$$\nu ([t -\lambda^{−3n} , t + \lambda^{−3n} ]) < \lambda^{-3n}$$
for all integers $n \geq 1$.
\end{lemma}
\begin{proof}
Let us partition the interval $I$ into ${\lambda^2}$ subintervals $J$
of length $|I|\lambda^{-2}$. 
Let $B_1$ be the family of all intervals $J$ from this partition for which
$\nu(J)<\lambda^{-1}$. Obviously, there are at most $\lambda$
intervals in $B_1^c$. Thus 
$$
\#B_1>\lambda^2-\lambda=\lambda^2\lt(1-\frac{\lambda}{\lambda^2}\rt)
$$
and
$$
{\rm Leb}\lt(\bigcup\{J:J\in B_1\}\rt)
\ge |I|\lt(1-{\lambda\over{\lambda^2}}\rt)
=|I|\lt(1-{1\over\lambda}\rt).
$$
Next, each interval in $B_1$ is divided into $\lambda^2$ subintervals
with disjoint interiors and of length $|I|\lambda^{-4}$, and
we remove those subintervals for which $\nu(J)\ge\lambda^{-2}$.
Denoting by $B_2$ the family of remaining intervals, we see that 
$$
\#B_2\ge (\lambda^2)^2\lt(1-{\lambda\over{\lambda^2}}\rt)-\lambda^2
=(\lambda^2)^2\lt(1-{1\over\lambda}-{1\over\lambda^2}\rt)
$$
and
$$
{\rm Leb}\lt(\bigcup\{J:J\in B_2\}\rt)
\ge |I|\lt(1-{1\over\lambda}-{1\over{\lambda^2}}\rt).
$$
Proceeding inductively, we partition the interval $I$ into disjoint intervals
of length $|I|\lambda^{-2n}$. Next,  we define
 in the same way the family   $B_n$. It is formed by the intervals
 $J$  of this partition of $n$'th generation, which are contained in
 some interval of the family $B_{n-1}$ and  for which
 $\nu(J)<\lambda^{-n}$. Then 
 $$
{\rm Leb}\lt(\bigcup\{J:J\in B_n\}\rt)
\ge\lt(1-{1\over\lambda}-{1\over{\lambda^2}}-\dots
 -{1\over{\lambda^n}}\rt)|I|.
$$ 
For any $t\in I$ let $J_n=J_n(t)$ be the interval of the
$n$'th  partition such that $t\in J_n$. Thus, for every 
$t\in\bigcap_{n=1}^\infty\bigcup_{J\in B_n}J$, we have that
$J_n(t)\in B_n$. Consequently, for all
$t\in\bigcap_{n=1}^\infty\bigcup_{J\in B_n}J$, it holds that
$\nu(J_n(t))<\lambda^{-n}$ for all $n\ge 1$.  Let now
$$
C_n=\{t\in I: [t-
|I|\lambda^{-3n},t+|I|\lambda^{-3n}]\subset J_n(t)\}.
$$
It is easy to see that ${\rm Leb}(C_n^c)<2|I|\lambda^{-n}$,
and, therefore, 
$$
{\rm Leb}\lt( \bigcap_{n=1}^\infty C_n\rt)
>|I|\lt(1-2\lt({1\over\lambda}+{1\over\lambda^2}+\dots\rt)\rt).
$$
Finally, setting
$$
I':= \lt(\bigcap_{n=1}^\infty C_n\rt)\cap\lt(
\bigcap_{i=1}^\infty\bigcup_{J\in B_i}J \rt)
$$ 
completes the proof. 
\end{proof}

\begin{proof} [Proof of Theorem \ref{main}] 
We can assume that $\mu(X)\leq 1$. We define finite Borel measures on the unit interval for all $z \in \supp \mu$ by
$$\mu_z(F)=\mu \{x \in X: d(x,z) \in F\}, \, F \subset [0,1].$$
Let $A_z= \cup_{\lambda >2} I'_z(\lambda)$ where $I'_z(\lambda)$ are the sets we obtain after we apply Lemma \ref{exp} to the measures $\mu_z$.
Then Lemma \ref{exp} implies that $\mu_z(A_z)=\mu_z([0,1])$. Let $G_z=\{r \in (0,1]:\ r \in A_z\}$ and 
$$\mathcal{X}_{B}'=\left\{\sum_{i=1}^n a_i\chi_{B(z_i,r_i)}: \, n \in \N, \, a_i\in \R, \, z_i \in \supp \mu, \, r_i \in G_{z_i}\right\}.$$
Then $\mathcal{X}_{B}'$ is dense in $L^2(\mu)$.

Let $f,g \in \mathcal{X}_{B}'$ such that
\begin{equation*}
f=\sum_i^n a_{i}\chi _{B_i}\text{ and }g=\sum_j^m b_{j}\chi _{S_{j}},
\end{equation*}
where $a_{i},b_{j}\in \R$ and $B_{i},S_{j}$ are closed balls. Then for $0<\delta<\varepsilon$,
\begin{equation*}
\int T_{\mu,\ve }^{k }f(x)g(x)d\mu(x)-\int T_{\mu,\delta }^{k }f(x)g(x)d\mu(x)=
\sum_{j=1}^m\sum_{i=1}^n  a_ib_j \underset{ \delta <d(x,y)<\varepsilon }{\int_{S_j}\int_{B_i}} k(x,y) d \mu (y) d \mu (x).  
\end{equation*}
Furthermore,
\begin{equation*}
\begin{split}
&\left| \underset{\delta <d(x,y) <\varepsilon
}{\int_{S_{j }}\int_{B_{i}}}k(x,y)d\mu (y)d\mu (x)\right| \\
& \quad \leq \left| \underset{\delta <d(x,y)<\varepsilon }{\int_{B_{i}\cap
S_{j}}\int_{B_{i}\cap
S_{j}}}k(x,y)d\mu (y)d\mu (x)\right| +\left| \underset{\delta <d(x,y)
 <\varepsilon }{\int_{S_{j}\setminus B_{i}}\int_{B_{i}\cap
S_{j}}}k(x,y)d\mu (y)d\mu (x)\right|\\ 
& \quad \quad \quad \quad + \left| \underset{\delta <d(x,y) <\varepsilon }{\int_{S_{j}\setminus
B_{i}}\int_{B_{i}\setminus S_{j}}}k(x,y)d\mu (y)d\mu (x) \right| +\left|
\underset{\delta <d(x,y) <\varepsilon }{\int_{S_{j}\cap
B_{i}}\int_{B_{i}\setminus S_{j}}}k(x,y)d\mu (y)d\mu (x) \right|\\
&\quad \leq \underset{\delta <d(x,y) <\varepsilon }{
\int_{B_{i}}\int_{B_{i}^{c}}}\left| k(x,y)\right| d\mu (y)d\mu
(x)+2\underset{\delta <d(x,y) <\varepsilon }{\int_{S_{j}}\int_{S_{j}^{c}}}\left| k(x,y)\right| d\mu (y)d\mu (x).
\end{split}
\end{equation*}
The last inequality follows because by antisymmetry and Fubini's theorem 
$$\underset{\delta <d(x,y)<\varepsilon }{\int_{B_{i}\cap
S_{j}}\int_{B_{i}\cap
S_{j}}}k(x,y)d\mu (y)d\mu (x)=0.$$
Therefore it is enough to show that for any ``good" ball $B=B(z,r)$ with $z \in \supp \mu$ and $r \in G_{z}$
$$\lim_{\substack{0<\delta<\ve \\ \ve \ra 0}} \underset{\delta <d(x,y) <\varepsilon }{\int_B \int_{B^c}}|k(x,y)| d\mu (y)d\mu (x)=0,$$
which will follow by the monotone convergence theorem if we show that 
\begin{equation}
\label{finite}
\int_B \int_{B^c}|k(x,y)| d\mu (y)d\mu (x)<\infty.
\end{equation}

Since $B=B(z,r)$ and $r \in G_z$ Lemma \ref{exp} implies that $\mu(\partial B)=0$ hence it is enough to show that
$$\int_{B^o} \int_{B^c}|k(x,y)| d\mu (y)d\mu (x)<\infty$$
where $B^o$ stands for the interior of $B$. For any $x \in B^o$ let $n(x)>0$ such that
$$2^{n(x)} d(x, \partial B)=3$$
and $N(x)=\text{integer part of }n(x)+1.$ Therefore, since $\diam (B) \leq 1$, $$B(x,2) \stm B \subset \cup_{i=1}^{N(x)} B(x,2^i d(x, \partial B))\setminus B(x,2^{i-1} d(x, \partial B)).$$ Hence for all $x \in B^o$
\begin{equation*}
\begin{split}
\int_{B(x,2) \stm B}|k(x,y)| d\mu (y) &\leq \int_{B(x,2) \stm B} d(x,y)^{-s} d\mu (y) \\
&= \sum_{i=1}^{N(x)} \int_{B(x,2^i d(x, \partial B))\setminus B(x,2^{i-1} d(x, \partial B))} d(x,y)^{-s}d \mu (y)\\
&\leq \sum_{i=1}^{N(x)} \mu(B(x,2^i d(x, \partial B)) (2^{i-1} d(x, \partial B))^{-s} d \mu (y) \\
& \lesssim N(x) \lesssim |\log d(x,\partial B)|,
\end{split}
\end{equation*}
and 
\begin{equation*}
\begin{split}
\int_{B^c}|k(x,y)| d\mu (y) &\lesssim \int_{B(x,2)^c}d(x,y)^{-s} d\mu (y)+ |\log d(x,\partial B)|\\
&\lesssim 1 + |\log d(x,\partial B)|.
\end{split}
\end{equation*}
Since $r \in G_z$ there exists some $\lambda \in \N$ such that $r \in I'_z(\lambda)$. We write,
\begin{equation*}
\begin{split}
\int_{B(z,r)^o}|\log d(x, \partial B)| d \mu (x)&=\int_{B(z,r-\lambda^{-3})^o}|\log d(x, \partial B)| d \mu (x)\\
& \quad \quad \quad+\sum_{n=1}^\infty \int_{\{x: r-\lambda^{-3n}\leq d(z,x)<r-\lambda^{-3(n+1)}\}}|\log d(x, \partial B)| d \mu (x)
\end{split}
\end{equation*}
Notice that by Lemma \ref{exp}
\begin{equation*}
\begin{split}\mu(\{x: r-\lambda^{-3n}&\leq d(z,x)<r-\lambda^{-3(n+1)}\})=\mu_z([r-\lambda^{-3n},r-\lambda^{-3(n+1)}))\\
&\quad  \quad \leq \mu_z([r-\lambda^{-3n},r+\lambda^{-3n}))\leq \lambda^{-n}.
\end{split}
\end{equation*}
Therefore,
\begin{equation*}
\begin{split}
\int_{B(z,r)^o}|\log d(x, \partial B)| d \mu (x) \lesssim 3\log(\lambda)(r-\lambda^{-3})^s+\sum_{i=1}^n \lambda^{-n} |\log (\lambda^{-3(n+1)})|<\infty
\end{split}
\end{equation*}
and this completes the proof of Theorem \ref{main}.
\end{proof}

\enlargethispage{2cm}

\end{document}